\documentclass[preprint]{elsarticle}

\usepackage{nccmath}
\usepackage{amsmath}
\usepackage{amsthm}
\usepackage{amssymb}
\usepackage{pdfpages}
\usepackage{afterpage}
\usepackage{bm}
\usepackage{mathtools}
\usepackage{tikz}
\usepackage{hyperref}

\theoremstyle{plain}
\newtheorem{thm}{Theorem}[section]
\newtheorem{lem}{Lemma}[section]
\newtheorem{prop}{Proposition}[section]

\theoremstyle{definition}
\newtheorem{rem}{Remark}[section]

\newcommand{\vu}{\bm{u}}
\newcommand{\vv}{\bm{v}}

\newcommand{\BR}{\mathbb{R}}
\newcommand{\BN}{\mathbb{N}}

\newcommand{\BC}{\mathbb{C}}

\newcommand{\CD}{\mathcal{D}}

\newcommand{\Tr}{\mathrm{tr}}
\newcommand{\RE}{\mathrm{Re}}

\newcommand{\Diag}{\mathrm{diag}}

\newcommand{\la}{\langle}
\newcommand{\ra}{\rangle}
\newcommand{\MBC}[1]{M_{#1}(\BC)}

\begin{document}
\begin{frontmatter}
\title{A note on some upper bounds for the Frobenius norm of the $q$-deformed commutator}

\author{Motoyuki NOBORI}
\address{Graduate School of Science and Engineering, Ehime University, 2-5 Bunkyo-cho, Matsuyama 790-8577, Japan}
\ead{m819002c@mails.cc.ehime-u.ac.jp}

\begin{abstract}
For two square complex matrices $A,B,$ and a real number $q$, the $q$-deformed commutator of $A$ and $B$ is defined as $AB-qBA$. We give some upper bounds on the Frobenius norm of the $q$-deformed commutator in the case $q\ge0$. All of the inequalities derived in this paper generalize the B\"{o}ttcher-Wenzel inequality.
\end{abstract}


\end{frontmatter}

\section{Introduction}\label{sec:intro}
 We write $M_{m,n}(\BC)$ for the vector space of all $m\times n$ complex matrices with the inner product $\la A,B\ra=\Tr(AB^{*})$, where $B^{*}=\overline{B^T}$ is the conjugate transpose of $B$ and $\Tr$ denotes the trace. The B\"{o}ttcher-Wenzel inequality gives the following upper bound on the Frobenius norm of the commutator of two square matrices $A$ and $B$:
\begin{align}\label{ineq:BWineq}
\|AB-BA\|_{F}^{2}\le 2\|A\|_{F}^{2}\|B\|_{F}^{2},
\end{align}
where $\|A\|_{F}=\sqrt{\la A,A\ra}$ is the Frobenius norm. Inequality (\ref{ineq:BWineq}) was conjectured to hold for all square real matrices $A$ and $B$ in 2005 \cite{HowBig}. After some proofs in the general case (cf. \cite{ProOf} for the real case and \cite{VarBou, TheFro} for the complex case), several subsequent problems have been considered (cf. \cite{Survey, Survey2}). For example, as an inequality whose upper bound is tighter than that of (\ref{ineq:BWineq}),
\begin{align}\label{ineq:GBWineq1}
\|AB-BA\|_{F}^{2}\le 2\|A\|_{(2),2}^{2}\|B\|_{F}^{2}    
\end{align}
holds \cite{VarBou}, where $\|A\|_{(2),2}=\sqrt{\sigma_{1}^2(A)+\sigma_{2}^2(A)}$ is the Ky Fan (2,2)-norm, and $\sigma_{1}(A)\ge \dots \ge\sigma_{n}(A)$ are its singular values. A characterization of its equality case is given in \cite{GBWineqCharact}. Some generalizations of (\ref{ineq:BWineq}) by changing commutators and norms have also been investigated (cf. \cite{Gen3mat, GenWeimat, qdefo1}). Especially in \cite{qdefo1}, sharp upper bounds on the Frobenius norm of the $q$-deformed commutator defined as $AB-qBA$ for a real number $q$ and two square complex matrices $A,B$ are analyzed: an inequality for $\|AB-qBA\|_{F}$ is called sharp at $q'$ if there exist non-zero matrices $A$ and $B$ satisfying its equality at the point $q=q'$. The following results are known \cite{qdefo1}:
\begin{enumerate}
    \item[\textup{(i)}] If $q\le 0$, then the sharp bound is given by
    \begin{align*}
        \|AB-qBA\|_{F}^{2} \le (1-q)^2\|A\|_{F}^{2}\|B\|_{F}^2.
    \end{align*}
    \item[\textup{(ii)}] If $q > 0$ and $A$ or $B$ is normal, then the sharp bound is given by 
    \begin{align}\label{ineq:qdefo}
        \|AB-qBA\|_{F}^{2} \le (1+q^2)\|A\|_{F}^{2}\|B\|_{F}^2.
    \end{align}
    Also, if $A$ is normal, then 
    \begin{align*}
        \|AB-qBA\|_{F}^{2} \le (1+q^2)\|A\|_{(2),2}^{2}\|B\|_{F}^2
    \end{align*}
    holds for all complex square matrices $B$.
\end{enumerate}
It is also reported in \cite{qdefo1} that inequality (\ref{ineq:qdefo}) fails for some $q \ge 0$ and non-normal matrices $A,B$. In this paper, we derive some upper bounds on the Frobenius norm of the $q$-deformed commutator in the case $q\ge 0$. 
\begin{thm}\label{thm:main2}
Let $n$ be a positive integer, $A,B$ be $n\times n$ complex matrices, and $q\ge0$. 
\begin{enumerate}
    \item[\textup{(i)}] If $q \in [0,1]$, then 
    \begin{align} \label{ineq:sqdefog1}
    \|AB-qBA\|_{F}^{2} &\le \{(1+q)\|A\|_{(2),2}^{2}-2q(1-q)\sigma_{n}^2(A)\}\|B\|_{F}^{2}.
    \end{align}
    \item[\textup{(ii)}] If $q \in [1,+\infty)$, then 
    \begin{align}\label{ineq:sqdefog2}
    \|AB-qBA\|_{F}^{2} &\le \{(q^2+q)\|A\|_{(2),2}^{2}-2(q-1)\sigma_{n}^2(A)\}\|B\|_{F}^{2}.
    \end{align}
\end{enumerate}
\end{thm}

Inequalities (\ref{ineq:sqdefog1}) and (\ref{ineq:sqdefog2}) are generalizations of (\ref{ineq:GBWineq1}), for both sides of these inequalities coincide with those of (\ref{ineq:GBWineq1}) if $q=1$. The above inequalities and their equality conditions will be derived in section \ref{sec:main}. We also give two other upper bounds on $\|AB-qBA\|_{F}$ that are generalizations of (\ref{ineq:GBWineq1}), and their equality conditions. Then we see that the upper bounds given in sections \ref{sec:intro} and \ref{sec:main} are sharp at some limited points $p$. In section \ref{sec:additional}, we explore the existence of upper bounds on $\|AB-qBA\|_{F}$ tighter than those in Theorem \ref{thm:main2}. \par
Here, we introduce the additional notation used in this paper. We abbreviate $M_{n,n}(\BC)$ and $M_{n,1}(\BC)$ to $M_{n}(\BC)$ and $\BC^{n}$, respectively. We denote by $I_{n}$ the identity matrix of order $n$, and we use $O_{n}$ to denote the zero matrix of order $n$. We write $\BR, \BN$ as the set of real and natural numbers, respectively. If $A, B \in M_{n}(\BC)$ are Hermitian, then $B\ge A$ means that $B-A$ is positive semidefinite.

\section{Main results}\label{sec:main}
We first prove Theorem \ref{thm:main2}. Let $n$ be a positive integer. We call a matrix $J \in \MBC{n}$ a density matrix if $J$ satisfies $\Tr(J)=1$ and $J\ge O_{n}$. In the proof of Theorem \ref{thm:main2}, proving that
\begin{align*}
    \Tr(J(AA^*+A^*A))-|\Tr(JA)|^2\le \|A\|_{(2),2}^2
\end{align*}
for all matrices $A\in M_{n}(\BC)$ and density matrices $J\in M_{n}(\BC)$ is the most important step. And its argument is guided by the proof of the following statement:  
\begin{align}\label{ineq:VarbouImportant}
    \Tr(J(AA^*+A^*A))-2|\Tr(JA)|^2\le \|A\|_{(2),2}^2
\end{align}
for all matrices $A\in M_{n}(\BC)$ and density matrices $J\in M_{n}(\BC)$ (\cite{VarBou}, Theorems 9 and 13). In the proof, we will use the following lemma.
\begin{lem}[\cite{fillmore1969similarity}, Corollary 1]\label{lem:trunisim}
A matrix with trace zero is unitarily equivalent to a matrix with main diagonal consisting of zeros.
\end{lem}
\begin{proof}[Proof of Theorem \textup{\ref{thm:main2}}]  If $n=1$, noting that the Ky Fan (2,2)-norm of a one-dimensional matrix is the absolute value of its element, we can easily verify that the inequalities in the theorem hold. Moreover, if $A$ or $B$ is the zero matrix, then $AB-qBA=O_{n}$ and inequalities (\ref{ineq:sqdefog1}) and (\ref{ineq:sqdefog2}) clearly hold. Hence, we consider the case where $n\ge2$ and $A,B\neq O_{n}$. Let $q\in [0,1].$ Taking the sum of 
\begin{align*}
    \|AB-qBA\|_{F}^2 =\Tr(A^*ABB^*+q^2AA^*B^*B-qABA^*B^*-qBAB^*A^*)
\end{align*}
and
\begin{align*}
    \|A^*B+qBA^*\|_{F}^2 =\Tr(AA^*BB^*+q^2A^*AB^*B+qABA^*B^*+qBAB^*A^*)
\end{align*}
gives 
\begin{align} \label{eq:comstep1}
   \|AB-qBA\|_{F}^2+ \|A^*B+qBA^*\|_{F}^2 =\Tr((BB^*+q^2 B^*B)(AA^*+A^*A)).
\end{align}
From the Cauchy-Schwarz inequality, we have 
\begin{align} \notag
   |\Tr((BB^*+qB^*B)A^*)|&=|\la A^*B+qBA^*, B\ra|\\ \notag
   &\le \|A^*B+qBA^*\|_{F}\|B\|_{F}.
\end{align}
It follows from the above inequality and (\ref{eq:comstep1}) that
\begin{align}\notag
\|AB-qBA\|_{F}^{2}& =\Tr((BB^*+q^2 B^*B)(AA^*+A^*A)) - \|A^*B+qBA^*\|_{F}^2 \\ \begin{split}\label{ineq:CS}              
                  &\le \Tr((BB^*+q^2 B^*B)(AA^*+A^*A)) \\&\quad \quad-\frac{|\Tr((BB^*+qB^*B)A^*)|^2}{\|B\|_{F}^2}\end{split}\\ \notag
                  &= \Tr((BB^*+q B^*B)(AA^*+A^*A)) -\frac{|\Tr((BB^*+qB^*B)A^*)|^2}{\|B\|_{F}^2}\\ \notag
                  &\quad-q(1-q)\Tr((AA^*+A^*A)B^*B)\\ \label{eq:putk1k2}
                  &\eqcolon k_{1}+k_{2},
\end{align}
where 
\begin{align*}
k_{1}&=\Tr((BB^*+qB^*B)(AA^*+A^*A)) -\frac{|\Tr((BB^*+qB^*B)A^*)|^2}{\|B\|_{F}^2}, \\
k_{2}&=-q(1-q)\Tr((AA^*+A^*A)B^*B).   
\end{align*}
We need to verify
\begin{align}\label{k1ub}
    k_{1}\le (1+q)\|A\|_{(2),2}^2\|B\|_{F}^2
\end{align}
and 
\begin{align}\label{k2ub}
    k_{2}\le -2q(1-q)\sigma_{n}^2(A)\|B\|_{F}^2
\end{align}
to obtain inequality (\ref{ineq:sqdefog1}). We first show inequality (\ref{k1ub}). Put $J_{1}=(BB^*+qB^*B)/\{(1+q)\|B\|_{F}^2\}$. Then $J_{1}$ is a density matrix (a positive semidefinite matrix of trace one), and we have 
\begin{align}\notag
    k_{1}&=(1+q)\|B\|_{F}^2\{\Tr(J_{1}(AA^*+A^*A))-(1+q)|\Tr(J_{1}A^*)|^2\} \\ \label{ineq:k1point1}
    &\le (1+q)\|B\|_{F}^2\{\Tr(J_{1}(AA^*+A^*A))-|\Tr(J_{1}A)|^2\}.
\end{align}
The last inequality holds due to $q\ge0$ and $|\Tr(J_{1}A^*)|=|\Tr(J_{1}A)|$. From (\ref{ineq:k1point1}), it suffices to show
\begin{align}\label{ineq:k1subgoal}
\Tr(J_{1}(AA^*+A^*A))-|\Tr(J_{1}A)|^2\le \|A\|_{(2),2}^2   
\end{align}
to have (\ref{k1ub}). Let us introduce some notions to prove (\ref{ineq:k1subgoal}). The symbol $\CD$ denotes the set of all density matrices in $\MBC{n}$. Notice that $\CD$ is convex, bounded, and closed. Let us denote by $\CD_{1}$ the set of all density matrices with rank 1 in $M_{n}(\BC)$. Note that $\CD_{1}$ is bounded and closed. For any matrix $X \in \MBC{n}$, $W(X)$ denotes the numerical range of $X$ defined as $\{\Tr(JX)|J\in \CD_{1}\}.$ We note that $W(X)$ is convex (\cite{RBMatAna}, the Toeplitz-Hausdorff theorem), bounded, and closed. The convexity of $W(X)$ yields $W(X)=\{\Tr(JX)|J\in\CD\}$. Here, we define the function $f_{A}:\CD\to\BR;\rho\mapsto\Tr(\rho(AA^{*}+A^{*}A))-|\Tr(\rho 
A)|^2$. From the definition, the left-hand side of (\ref{ineq:k1subgoal}) can be written as $f_{A}(J_{1})$. Since $f_{A}$ is continuous on both $\CD$ and $\CD_{1}$, the extreme value theorem implies that $f_{A}$ takes on its maximum values on both sets. From this and the fact that $\CD_{1}\subset \CD$, we have 
\begin{align}\label{ineq:D1D}
\max_{\rho \in \CD_{1}}f_{A}(\rho)\le \max_{\rho\in \CD}f_{A}(\rho).
\end{align}
We shall prove that $\max_{\rho \in \CD_{1}}f_{A}(\rho) \ge \max_{\rho\in \CD}f_{A}(\rho)$. Combined with the previous argument, this yields $\max_{\rho \in \CD_{1}}f_{A}(\rho)= \max_{\rho\in \CD}f_{A}(\rho)$. Choose $\rho'\in \CD$ such that $f_{A}(\rho')=\max_{\rho \in \CD}f_{A}(\rho)$. Here, let $c=\Tr(\rho'A)$. Then we have $\Tr(\rho'^{1/2}(A-cI_{n})\rho'^{1/2})=0$. By this identity and Lemma \ref{lem:trunisim}, there exists a unitary matrix $U\in M_{n}(\BC)$ such that all the diagonal entries of $U^{*} \rho'^{1/2}(A-cI_{n})\rho'^{1/2}U$ have the same value 0. Put 
$\begin{bmatrix}\bm{p}_{1},\dots,\bm{p}_n\end{bmatrix}\coloneq \rho'^{1/2}U$. After deleting the zero vectors from $\bm{p}_{1},\dots,\bm{p}_n$, we denote the remaining vectors by $\bm{p}'_{1},\dots ,\bm{p}_{k}'$, where $k\in\{1,\dots,n\}$. We also put $\vu_{i}\coloneq \bm{p}'_{i}/\|\bm{p}'_{
i}\|_{F}, \alpha_{i}\coloneq\|\bm{p}'_{i}\|_{F}$ for $i\in \{1,\dots ,k\}$. From these definitions and the fact that $\rho'\in \CD$, we obtain 
\begin{align*}
    &\vu_{i}\vu_{i}^*\in \CD_{1} \ \ \text{for all $i\in\{1,\dots ,k\}$,}\\
&\sum_{i=1}^{k}\alpha_{i}^2=\sum_{i=1}^{n}\|\bm{p}_{i}\|_{F}^2=1,\\ &\sum_{i=1}^{k}\alpha_{i}^2\vu_{i}\vu^{*}_{i}=\sum_{i=1}^{n}\bm{p}_{i}\bm{p}^{*}_{i}=\rho'
.\end{align*}
Moreover, for $\vu_{1},\dots,\vu_{k}$ and the diagonal elements of $U^{*} \rho'^{1/2}(A-cI_{n})\rho'^{1/2}U$, we see that
\begin{align*}
    &\Tr(\vu_{i}\vu_{i}^* A) =\vu_{i}^* A\vu_{i}=\Tr(\rho'A)\ \ \text{for all $i\in\{1,\dots ,k\}$}.\end{align*}
For $f_{A}(\rho')$, we have 
\begin{align*}
    f_{A}(\rho')&=\Tr(\rho'(AA^*+A^*A))-|\Tr(\rho'A)|^2\\
    &=\sum_{i=1}^{k}\alpha_{i}^2\Tr(\vu_{i}\vu_{i}^*(AA^*+A^*A))-\sum_{j=1}^{k}\alpha_{j}^2|\Tr(\vu_{j}\vu_{j}^{*}A)|^2\\
    &=\sum_{i=1}^{k}\alpha_{i}^2f_{A}(\vu_{i}\vu_{i}^*)\\
    &\le \max_{\rho\in\CD_{1}}f_{A}(\rho).
\end{align*}
As for the left-hand side of (\ref{ineq:k1subgoal}), the above inequality and (\ref{ineq:D1D}) lead to \begin{align*}
    f_{A}(J_{1})\le \max_{\rho \in \CD}f_{A}(\rho)=\max_{\rho \in \CD_{1}}f_{A}(\rho).
\end{align*} Suppose that $\rho_{1}'\in \CD_{1}$ satisfies $f_{A}(\rho_{1}')=\max_{\rho \in \CD_{1}}f_{A}(\rho)$. Then there exists a unit vector $\vu\in \BC^{n}$ such that $\rho_{1}'=\vu\vu^*$. Hence, $f_{A}(\rho_{1}')$ can be expressed as 
\begin{align}\label{eq:thmpoint3}
    f_{A}(\rho_{1}')=\la\vu,(AA^*+A^*A)\vu\ra-|\la \vu, A\vu\ra|^2.
\end{align}
We next prove that  
\begin{align*}
    \la\vu, (AA^*+A^*A)\vu\ra-|\la\vu,A\vu\ra|^2\le \|A\|_{(2),2}^2
\end{align*}
by an argument similar to the proof of Theorem 13 in \cite{VarBou}, and we finally obtain inequality (\ref{ineq:k1subgoal}). Let us explain in detail. Choose two orthonormal bases $\{\vu_{i}\}$ and $\{\vv_{i}\}$ of $\BC^n$ so that $\vu_{1}=\vv_{1}=\vu$ and $\la A\vv_{1},\vu_{i}\ra=0$ for $\forall i>2$. Put $U=[\vu_{1},\dots,\vu_{n}], V=[\vv_{1},\dots,\vv_{n}]$, and $A=UU^*AVV^*\eqcolon U[a'_{i,j}]V^*$, where $a'_{i,j}=\la A\vv_{j},\vu_{i} \ra$. Then the right-hand side of (\ref{eq:thmpoint3}) is bounded above by
\begin{align*}
    \la\vu, (AA^*+A^*A)\vu\ra-|\la\vu,A\vu\ra|^2&=\la\vu_{1}, AA^* \vu_{1}\ra +\la \vv_{1},A^*A\vv_{1}\ra-|\la\vu_{1},A\vv_{1}\ra|^2\\ 
    &=\sum_{j=1}^{n}(\la\vu_{1}, A\vv_{j}\ra\la\vv_{j}, A^* \vu_{1}\ra\\
    &\quad+ \la\vv_{1}, A^* \vu_{j}\ra\la\vu_{j}, A \vv_{1}\ra) -|a'_{1,1}|^2\\
    &=\sum_{j=1}^{n}(|a'_{j,1}|^2+|a'_{1,j}|^2) -|a'_{1,1}|^2\\
    &=|a'_{2,1}|^2+\sum_{j=1}^{n}|a'_{1,j}|^2\\
    &\le \sum_{j=1}^n (|a'_{1,j}|^2+|a'_{2,j}|^2)\\
    &=\Tr(U^*AVV^*A^*U\Diag\{1,1,0,\dots,0\})\\
    &=\Tr(AA^*[\vu_{1},\vu_{2}][\vu_{1},\vu_{2}]^*)\\
    &\le \|A\|_{(2),2}^2.
\end{align*}
The last inequality holds because of the orthonormality of $\vu_{1}, \vu_{2}$ and Ky Fan's maximum principle \cite{RBMatAna}. We conclude from the above inequality that (\ref{k1ub}) holds. Let us prove inequality (\ref{k2ub}). Letting $J_{2}=B^*B/\|B\|_{F}^2$ we see that $J_{2}$ is a density matrix and $k_{2}$ can be written as $k_{2}=-q(1-q)\|B\|_{F}^2\Tr(J_{2}(AA^*+A^*A))$. We obtain from $\Tr(J_{2}(AA^*+A^*A))\in W(AA^{*}+A^{*}A)=[\sigma_{n}(AA^*+A^*A),\sigma_{1}(AA^*+A^*A)]$ and Weyl's inequality \cite{RBMatAna} that
\begin{align*}
-q(1-q)\Tr(J_{2}(AA^*+A^*A)) &\le-q(1-q)\sigma_{n}(AA^*+A^*A)\\
&\le-q(1-q)(\sigma_{n}(AA^*)+\sigma_{n}(A^*A))\\
&=-2q(1-q)\sigma_{n}^2(A),
\end{align*}
which gives (\ref{k2ub}). We finally obtain inequality (\ref{ineq:sqdefog1}).\par
Through a similar discussion by replacing $\|A^*B+qBA^*\|_{F}^2$ by $\|AB^*+qB^*A\|_{F}^2$ and $|\Tr((BB^*+qB^*B)A^*)|$ by $|\Tr((A^*A+qAA^*)B^*)|$, we also obtain 
\begin{align}\label{ineq:qdefog1BA}
    \|AB-qBA\|_{F}^{2}&\le \{(1+q)\|B\|_{(2),2}^2-2q(1-q)\sigma_{n}^2(B)\}\|A\|_{F}^2.
\end{align}
When $q\ge1$, inequality (\ref{ineq:sqdefog2}) is derived from (\ref{ineq:qdefog1BA}), for (\ref{ineq:qdefog1BA}) and $\frac{1}{q}\in (0,1]$ yields  
\begin{align*}
    \|AB-qBA\|_{F}^{2}&=q^{2}\|BA-\frac{1}{q}AB\|_{F}^{2}\\ 
                      &\le q^2\{(1+\frac{1}{q})\|A\|_{(2),2}^2-2\frac{1}{q}(1-\frac{1}{q})\sigma_{n}^2(A)\}\|B\|_{F}^2\\
                      &=\{(q^2+q)\|A\|_{(2),2}^{2}-2(q-1)\sigma_{n}^2(A)\}\|B\|_{F}^{2},
\end{align*}
and we finally obtain (\ref{ineq:sqdefog2}). This completes the proof of Theorem \ref{thm:main2}. 
\end{proof}
And from now on, we shall consider the only case $q\in[0,1]$ because it follows from the above discussion of deducing (\ref{ineq:sqdefog2}) form (\ref{ineq:sqdefog1}) that the case $q\in (0,1]$ can be applied to the case $q\in[1,+\infty)$.
We also give an equivalent condition for the equality in (\ref{ineq:sqdefog1}).
\begin{prop}\label{prop:charact2}
Let $A,B\in M_{n}(\BC)\setminus\{O_{n}\}$ and $q\in[0,1]$.
\begin{enumerate}
    \item[\textup{(i)}] If $n\ge 2$, then the equality in $(\ref{ineq:sqdefog1})$ holds if and only if $q,A,$ and $B$ satisfy one of the following conditions.
    \begin{enumerate}
        \item[\textup{(a)}] $q=0$, and $\|AB\|_{F}^{2}=\|A\|_{F}^{2}\|B\|_{F}^2$.
        \item[\textup{(b)}] $q=1$, and $\|AB-BA\|_{F}^{2}=2\|A\|_{(2),2}^{2}\|B\|_{F}^2$.
    \end{enumerate}
    \item[\textup{(ii)}] If $n=1$, then the equality in \textup{(\ref{ineq:sqdefog1})} holds if and only if $q=0$.
\end{enumerate}
\end{prop}
\begin{proof}
Let $n\ge 2$. Since the sufficiency is obvious, we show the necessity. Suppose that the equality in (\ref{ineq:sqdefog1}) holds. We shall prove by contradiction that $q$ equals 0 or 1. This, together with the equality in (\ref{ineq:sqdefog1}), leads to condition \textup{(a)} or \textup{(b)}. As for the condition of $A$ and $B$ in \textup{(a)}, it is clear that $\|AB\|_{F}^2=\|A\|_{(2),2}^2\|B\|_{F}^2$ if and only if $\|AB\|_{F}^2=\|A\|_{F}^2\|B\|_{F}^2$. Suppose that $q\in (0,1)$. The equality in (\ref{ineq:sqdefog1}) implies that the equalities in (\ref{ineq:CS}), (\ref{k1ub}), (\ref{k2ub}), and (\ref{ineq:k1point1}) hold. It follows from the equality in (\ref{ineq:k1point1}) that $q|\Tr(J_{1}A)|^2=0$, which implies $\Tr(J_{1}A^*)=0$. Due to the equality in (\ref{ineq:CS}), there exists a number $\alpha\in\BC$ such that $A^*B+qBA^*=\alpha B$. This equation and $0=\Tr(J_{1}A^*)=\la A^*B+qBA^*,B\ra/\{(1+q)\|B\|_{F}^2\}$ imply $\alpha=0$.
By $A^*B=-qBA^*$ and the equalities in (\ref{k2ub}), (\ref{k1ub}), we see that  
\begin{align*}
    \sigma_{n}^2(A)\|B\|_{F}^2&=\Tr(B^*BAA^*)=\Tr(B^*BA^*A)=\frac{-1}{q}\Tr(B^*A^*BA),\\
    \Tr(BB^*AA^*)&=-q\Tr(BAB^*A^*)=q^2\sigma_{n}^2(A)\|B\|_{F}^2,
\end{align*}
and
\begin{align*}
    \|A\|_{(2),2}^2&=\Tr(J_{1}(AA^*+A^*A))\\
&= \frac{1}{(1+q)\|B\|_{F}^2}\{\Tr(BB^*A^*A)+(2q+q^2)\sigma_{n}^2(A)\|B\|_{F}^2\}\\
&\le \frac{1}{(1+q)\|B\|_{F}^2}\{\sigma_{1}^2(A)\|B\|_{F}^2+(2q+q^2)\sigma_{n}^2(A)\|B\|_{F}^2\}\\
&=\frac{1}{1+q}\sigma_{1}^2(A)+\frac{2q+q^2}{1+q}\sigma_{n}^2(A) \le \|A\|_{(2),2}^2.
\end{align*}
The above inequality implies that $\sigma_{1}^2(A)+(2q+q^2)\sigma_{n}^2(A)=(1+q)(\sigma_{1}^2(A)+\sigma_{2}^2(A))$. If $\sigma_{n}(A)=0$, then we have $0=q\sigma_{1}^2(A)+(1+q)\sigma_{2}^2(A)$, which leads to $q=0$, but this contradicts $q\neq0$. Suppose $\sigma_{n}(A)>0$. Then we see that $0=q(\sigma_{1}^2(A)-\sigma_{n}^2(A))+(1+q)(\sigma_{2}^2(A)-q\sigma_{n}^2(A))$. Since $q\neq 0$, this equation implies that $\sigma_{1}(A)=\sigma_{n}(A)$ and $\sigma_{2}^2(A)=q\sigma_{n}^2(A)$. From these equalities, we have $q=1$, but this contradicts $q\neq1$. Therefore, $q$ is 0 or 1. This completes the proof of case \textup{(i)}. We can show case \textup{(ii)} by a direct calculation.
\end{proof} 
From the above proposition, inequality (\ref{ineq:sqdefog1}) becomes sharp only at $q=0,1$. Next, we derive another upper bound on $k_{1}$ in (\ref{eq:putk1k2}) to obtain generalizations of (\ref{ineq:GBWineq1}) as follows:
\begin{thm}\label{thm:main1}
Let $n$ be a positive integer, $A,B\in \MBC{n}$, and $q \ge 0$. 
\begin{enumerate}
    \item[\textup{(i)}] If $q \in [0,1]$, then 
    \begin{align}
    \begin{split}
        \|AB-qBA\|_{F}^{2} &\le \{(1-q^2)\sigma_{1}^{2}(A)+\frac{(1+q)^2}{2}\|A\|_{(2),2}^{2}\\&\quad-2q(1-q)\sigma_{n}^2(A)\}\|B\|_{F}^{2}.
    \end{split}\label{ineq:qdefog1}
    \end{align}
    \item[\textup{(ii)}] If $q \in [1,+\infty)$, then 
    \begin{align*}
    \|AB-qBA\|_{F}^{2} &\le \{(q^2-1)\sigma_{1}^{2}(A)+\frac{(1+q)^2}{2}\|A\|_{(2),2}^{2}\\
    &\quad-2(q-1)\sigma_{n}^2(A)\}\|B\|_{F}^{2}.
    \end{align*}
\end{enumerate}
\end{thm}
\begin{proof}
We only show case \textup{(i)}. Let $A,B\neq O_{n}$. Since $\|AB-qBA\|_{F}^2\le k_{1}+k_{2}$, inequality (\ref{k2ub}), and $k_{1}=(1+q)\|B\|_{F}^2\{\Tr(J_{1}(AA^*+A^*A))-(1+q)|\Tr(J_{1}A^*)|^2\}$, it suffices to show that
\begin{align*}
    \{\Tr(J_{1}(AA^*+A^*A))-(1+q)|\Tr(J_{1}A^*)|^2\} \le (1-q)\sigma_{1}(A)^2+\frac{1+q}{2}\|A\|_{(2),2}^2.
\end{align*}
to get (\ref{ineq:qdefog1}). Inequality (\ref{ineq:VarbouImportant}) leads to
\begin{align*}
    \{\Tr(&J_{1}(AA^*+A^*A))-(1+q)|\Tr(J_{1}A^*)|^2\}\\
    &=\frac{1-q}{2}\Tr(J_{1}(AA^*+A^*A))+\frac{1+q}{2}\{\Tr(J_{1}(AA^*+A^*A))-2|\Tr(J_{1}A)|^2\} \\
    &\le\frac{1-q}{2}\Tr(J_{1}(AA^*+A^*A))+\frac{1+q}{2}\|A\|_{(2),2}^2\\
    &\le(1-q)\sigma_{1}^2(A)+\frac{1+q}{2}\|A\|_{(2),2}^2.
\end{align*}
The last inequality holds because $\Tr(J_{1}(AA^*+A^*A))\in W(AA^*+A^*A)=[\sigma_{n}(AA^*+A^*A), \sigma_{1}(AA^*+A^*A)]$ and Weyl's inequality. This completes the proof of Theorem \ref{thm:main1}.
\end{proof}
\begin{rem}\label{rem:hikaku}
Comparing (\ref{ineq:sqdefog1}) and (\ref{ineq:qdefog1}), we see that the right-hand side of (\ref{ineq:sqdefog1}) is tighter than that of (\ref{ineq:qdefog1}). Indeed, for any $A\in M_{n}(\BC)$ and $q\in[0,1]$, the difference of \begin{align*}h_{1}(q,A)\coloneq(1-q^2)\sigma_{1}^{2}(A)+\frac{(1+q)^2}{2}\|A\|_{(2),2}^{2}-2q(1-q)\sigma_{n}^2(A)
\end{align*}and 
\begin{align*}h_{2}(q,A)\coloneq(1+q)\|A\|_{(2),2}^{2}-2q(1-q)\sigma_{n}^2(A)\end{align*}
yields 
\begin{align*}
    h_{1}(q,A)-h_{2}(q,A)=\frac{1-q^2}{2}(\sigma_{1}^2(A)-\sigma_{2}^2(A))\ge0.
\end{align*}
\end{rem}
From Remark \ref{rem:hikaku} and Proposition \ref{prop:charact2}, we immediately obtain the following equality conditions of the inequalities in Theorem \ref{thm:main1}.
\begin{prop}\label{prop:charact1} Let $A,B\in \MBC{n}\setminus \{O_{n}\}$ and $q\in [0,1]$.
\begin{enumerate}
    \item[\textup{(i)}] If $n\ge 2$, then the equality in \textup{(\ref{ineq:qdefog1})} holds if and only if $q$, $A,$ and $B$ satisfy $q=1$, and $\|AB-BA\|_{F}^{2}=2\|A\|_{(2),2}^2\|B\|_{F}^2$.
    \item[\textup{(ii)}] If $n=1$, then the equality in \textup{(\ref{ineq:qdefog1})} never holds for any $q\in [0,1]$. 
\end{enumerate}
\end{prop}

\section{Some investigations on the sharp upper bound} \label{sec:additional}
Let $q\in[0,1]$. From (\ref{ineq:sqdefog1}) and the inequality $-2q(1-q)\sigma_{n}^2(A)\le 0$, we see that 
\begin{align*}
    \|AB-qBA\|_{F}^2\le (1+q)\|A\|_{(2),2}^2\|B\|_{F}^2.
\end{align*}
By Proposition \ref{prop:charact2}, we can easily confirm that this inequality is sharp only at $q=0,1$, the limit cases. Then let us explore the existence of a function $c_{q}\in [0,+\infty)$ on $q\in [0,1]$ satisfying all the following conditions:
\begin{itemize}
    \item[\textup{(I)}] \label{ineq:idealcq}$\|AB-qBA\|_{F}^{2}\le c_{q}\|A\|_{(2),2}^2\|B\|_{F}^2$ holds for $\forall A,B\in M_{n}(\BC)$.
    \item[\textup{(I\hspace{-1.2pt}I)}] The upper bound $c_{q}\|A\|_{(2),2}^2\|B\|_{F}^2$ becomes sharp at every $q\in[0,1]$.
\end{itemize}
The sharp $c_{q}$ has not yet been obtained. As an option of $c_{q}$, $1+q^l$ with $l\in(1,2)$ is possible because it gives the sharp upper bound when $q=0,1$ and satisfies $q^2\le q^l\le q$ for all $l\in(1,2)$, $q\in[0,1]$. However, as stated in the following remark, $c_{q}=1+q^l$ does not satisfy condition  \textup{(I)}. 
\begin{rem}\label{rem:invalidcoeff}
For any $l \in (1,2)$, we can confirm the existence of $q\in (0,1)$ and $A,B\in \MBC{2}\setminus \{O_{2}\}$ such that $\|AB-qBA\|_{F}^2>(1+q^l)\|A\|_{F}^2\|B\|_{F}^2$ as follows: since $1<l$, there exists a positive integer $k$ satisfying $1+10^{-k}<l$. Taking \begin{align*}
q=10^{-10^k}, A=\begin{bmatrix}0&-1\\0&-1\end{bmatrix}, B=\begin{bmatrix}0&0\\-2&1\end{bmatrix},
\end{align*} we get $\|AB-qBA\|_{F}^{2}=10+2q+q^2$ and $(1+q^l)\|A\|_{F}^{2}\|B\|_{F}^2=10+10q^l$. Here, since
\begin{align*}
2q+q^2&>q=10^{-10^k}=10(10^{-10^k-1})>10q^l,
\end{align*}
we obtain $\|AB-qBA\|_{F}^2>(1+q^l)\|A\|_{F}^2\|B\|_{F}^2$.
\end{rem}
In regard to Remark \ref{rem:invalidcoeff}, due to the way of choosing $q$, the matrices $A,B$ and the number $q$ in this remark satisfy the same inequality if we extend the range of $l$ from $(1,2)$ to $(1,+\infty)$. Therefore, we see that for any $d_{q}\in \{\Sigma_{i=1}^{m}a_{i}q^{b_{i}}|m\in \BN, b_{i}>1, a_{i}\ge0 \ (1\le \forall i \le m), \Sigma_{j=1}^{m}a_{j}=1\}\eqcolon \mathcal{P}_{q}$, there exist $q\in[0,1]$ and $A,B \in M_{n}(\BC)\setminus \{O_{n}\}$ with $n\ge 2$ such that 
\begin{align*}
    \|AB-qBA\|_{F}^{2}> (1+d_{q})\|A\|_{(2),2}^2\|B\|_{F}^2.
\end{align*}
Next, let us see the case where $c_{q}=1+aq+(1-a)d_{q}$ for some $a \in (0,1]$ and $d_{q}\in \mathcal{P}_{q}$. For example, if the dimension of matrices is 2 and $d_{q}=q^2$, then the following inequality holds:
\begin{prop}\label{prop:tightest2dim}
Let $a=\sqrt{2}-1$, $q\in[0,1]$, and $A,B\in M_{2}(\BC)$. Then we have
\begin{align}\label{ineq:propnotsharp}
    \|AB-qBA\|_{F}^{2}\le \{1+aq+(1-a)q^2\}\|A\|_{F}^{2}\|B\|_{F}^{2}.
\end{align}
Under the assumption $A,B\neq O_{2}$, the equality holds if and only if $q,A,B$ satisfy condition 
\textup{(a)} or \textup{(b)} in Proposition \textup{\ref{prop:charact2}}.
\end{prop}
\begin{proof}
Let $A=UD_{A}V$ be the singular value decomposition of $A$ with the diagonal matrix $D_{A}=\Diag\{\sigma_{1}(A),\sigma_{2}(A)\}$. Since the unitary invariance of the Frobenius norm, we have 
\begin{align} \notag
    \|AB-qBA\|_{F}^{2}&=\|D_{A}VBV^*-qU^*BUD_{A}\|_{F}^2\\ \notag
                      &\eqcolon\|D_{A}X-qYD_{A}\|_{F}^2\\ \label{ineq:sec3propbrbr}
                      &=\Sigma_{i,j}|\sigma_{i}(A)x_{i,j}-q\sigma_{j}(A)y_{i,j}|^2
\end{align}
where $VBV^*=X=[x_{i,j}]_{1\le i,j\le 2}$ and $U^*BU=Y=[y_{i,j}]_{1\le i,j\le 2}$. Note that $X$ and $Y$ are unitarily similar. As for (\ref{ineq:sec3propbrbr}), it follows that  \begin{align*}
&\Sigma_{i,j}|\sigma_{i}(A)x_{i,j}-q\sigma_{j}(A)y_{i,j}|^2\\
&\quad=\Sigma_{i}\sigma_{i}^2(A)|x_{i,i}-qy_{i,i}|^2+|\sigma_{1}(A)x_{1,2}-q\sigma_{2}(A)y_{1,2}|^2+|\sigma_{2}(A)x_{2,1}-q\sigma_{1}(A)y_{2,1}|^2\\
&\quad\le\Sigma_{i}\sigma_{i}^2(A)|x_{i,i}-qy_{i,i}|^2+(\sigma_{1}^2(A)+\sigma_{2}^2(A))\{|x_{1,2}|^2+|x_{2,1}|^2+q^2(|y_{1,2}|^2+|y_{2,1}|^2)\}.
\end{align*}
Let us analyze $|x_{i,i}-qy_{i,i}|^2$ with $i=1$. The case $i=2$ follows analogously. Due to the equality $x_{1,1}+x_{2,2}=\Tr(X)=\Tr(Y)=y_{1,1}+y_{2,2}$, we have 
\begin{align}\notag
    |x_{1,1}-qy_{1,1}|^2&=|x_{1,1}|^2+q^2|y_{1,1}|^2-2q\RE(x_{1,1}\overline{y_{1,1}})\\ \notag
    &=|x_{1,1}|^2+q^2|y_{1,1}|^2-2q\RE((y_{1,1}+y_{2,2}-x_{2,2})\overline{y_{1,1}})\\ \notag
    &=|x_{1,1}|^2+(q^2-2q)|y_{1,1}|^2+2q\RE(x_{2,2}\overline{y_{1,1}}-y_{2,2}\overline{y_{1,1}})\\ \label{ineq:amgm1}
    &\le |x_{1,1}|^2+|x_{2,2}|^2+2(q^2-q)|y_{1,1}|^2-2q\RE(y_{2,2}\overline{y_{1,1}}).
\end{align}
We shall show that
\begin{align}\label{ineq:propimppt}
    \{aq+(1-a)q^2\}(|y_{1,1}|^2+|y_{2,2}|^2)\ge 2(q^2-q)|y_{1,1}|^2-2q\RE(y_{2,2}\overline{y_{1,1}}).
\end{align}
Inequalities (\ref{ineq:propimppt}) with $i=1,2$ will yield the desired inequality, because these lead to
\begin{align}
\begin{split}\label{ineq:sec3proptotyu}
    \|AB-qBA\|_{F}^{2}&\le (\sigma_{1}^2(A)+\sigma_{2}^2(A))\{\|X\|_{F}^2 + q^2\|\begin{bmatrix}0&y_{1,2}\\ y_{2,1}&0\end{bmatrix}\|_{F}^2\\&\quad+\{aq+(1-a)q^2\}\|\begin{bmatrix}y_{1,1}&0\\ 0&y_{2,2}\end{bmatrix}\|_{F}^2\} 
\end{split}\\ \notag
    &\le \{1+aq+(1-a)q^2\}\|A\|_{F}^{2}\|B\|_{F}^{2}.
\end{align}
Inequality (\ref{ineq:propimppt}) is equivalent to 
\begin{align*}
    q\cdot h(q,y_{1,1},y_{2,2})\coloneq q\bigl(\{(a+1)(1-q)+1\}|y_{1,1}|^2+\{a+(1-a)q\}|y_{2,2}|^2+2\RE(y_{2,2}\overline{y_{1,1}})\bigl)\ge 0.
\end{align*}
If $\RE(y_{2,2}\overline{y_{1,1}})\ge 0$, then clearly $qh(q,y_{1,1},y_{2,2})\ge0$ holds. Suppose $\RE(y_{2,2}\overline{y_{1,1}})<0$.
Notice that  $h(q,y_{1,1},y_{2,2})$ is expressed as 
\begin{align}\notag
    h(q,y_{1,1},y_{2,2})&=|\sqrt{\{(a+1)(1-q)+1\}}y_{1,1}+\sqrt{\{a+(1-a)q\}}y_{2,2}|^2\\ \notag
    &\quad-2\bigl(\sqrt{\{(a+1)(1-q)+1\}\{a+(1-a)q\}}-1 \bigl)\RE(y_{2,2}\overline{y_{1,1}}).
\end{align}
With regard to the factor multiplied by $-2\RE(y_{2,2}\overline{y_{1,1}})$ in the above equality,
a direct calculation and the substitution of $\sqrt{2}-1$ for $a$ imply
\begin{align}\notag
\{(a+1)(1-q)+1\}\{a+(1-a)q\}&=-(1-a^2)(q+\frac{a^2+a-1}{1-a^2})^2 +\frac{1}{1-a^2}\\ \notag
&=-2(\sqrt{2}-1)(q-\frac{1}{2})^2+\frac{\sqrt{2}+1}{2}\ge1
\end{align}
for $\forall q\in[0,1]$. This gives $-2\bigl(\sqrt{\{(a+1)(1-q)+1\}\{a+(1-a)q\}}-1 \bigl)\RE(y_{2,2}\overline{y_{1,1}})\ge 0$, and we finally obtain $qh(q,y_{1,1},y_{2,2})\ge0$. Let us derive the equality condition. We only show the necessity. Suppose that the equality in (\ref{ineq:propnotsharp}) holds. Following Proposition \ref{prop:charact2}, we shall show $q$ equals $0$ or $1$. Since the equalities in (\ref{ineq:amgm1}) with $i=1,2$, we see that $x_{1,1}=qy_{2,2}$ and $x_{2,2}=qy_{1,1}$, which gives $\Tr(X)=q\Tr(Y)$. If $\Tr(X)\neq0$, then it follows from $\Tr(Y)=\Tr(X)=q\Tr(Y)$ that $q=1$. Suppose $\Tr(X)=0$. Then we have $y_{1,1}=-y_{2,2}$ and $\RE(y_{1,1}\overline{y_{2,2}})=-|y_{1,1}|^2\le 0$. If $\RE(y_{1,1}\overline{y_{2,2}})=0$, then we have $y_{1,1}=0$, and this implies $x_{1,1},x_{2,2},y_{2,2}=0$. Thus the left-hand side of (\ref{ineq:sec3proptotyu}) is bounded above by $(1+q^2)\|A\|_{F}^2\|B\|_{F}^2$. Therefore, we have 
\begin{align*}
    \{1+aq+(1-a)q^2\}\|A\|_{F}^2\|B\|_{F}^2&=\|AB-qBA\|_{F}^2\le(1+q^2)\|A\|_{F}^2\|B\|_{F}^2 \\
   &\le \{1+aq+(1-a)q^2\}\|A\|_{F}^2\|B\|_{F}^2.
\end{align*}
By virtue of the above inequality and $A,B\neq O_{n}$, we get $1+q^2=1+aq+(1-a)q^2$, which leads to $aq(1-q)=0$. This equalities shows $q$ is $0$ or $1$. Assume that $\RE(y_{1,1}\overline{y_{2,2}})<0$, then $h(q,y_{1,1},y_{2,2})$ must be zero. From this we see that $\sqrt{\{(a+1)(1-q)+1\}\{a+(1-a)q\}}=1$, which implies $q$ is equal to $0$ or $1$. This completes the proof.
\end{proof}
Note that $1+(\sqrt{2}-1)q+(2-\sqrt{2})q^2 \le 1+q$ holds for all $q\in[0,1]$, and inequality (\ref{ineq:propnotsharp}) becomes sharp only at $q=0,1$.
\section*{Acknowledgment} 
The author expresses his gratitude to Professor Masaya MATSUURA for his valuable comments which made the construction of this paper better. The author also expresses his gratitude to J\'{e}r\^{o}me NICOLAS for many valuable comments on the original paper, including a solution that fills the gap in the proof of Theorem 1.1.

\section*{Declaration of generative AI and AI-assisted technologies in the manuscript preparation process}
The author acknowledges the use of ChatGPT (GPT-5.6) for English-language editing. The author reviewed and edited the suggested changes and takes full responsibility for the content of the publication.
\bibliographystyle{elsarticle-num}
\bibliography{reference}
\end{document}